\documentclass{amsart}
\usepackage{amsmath,amsthm}
\usepackage{verbatim}
\usepackage{marvosym}
\usepackage{graphicx}
\usepackage{color}
\usepackage{epsfig}
\usepackage{ulem}
\usepackage{rotating}
\usepackage{lscape}
\usepackage{hyperref}
\usepackage[all]{xy}
\usepackage{soul}
\usepackage{graphics}
\usepackage{epstopdf}
\usepackage{tikz}
\usepackage{float}

\usepackage{caption}
\usepackage[labelformat=simple]{subcaption}

\theoremstyle{plain}
\newtheorem{thm}{Theorem}[section]
\newtheorem{cor}[thm]{Corollary}

\newtheorem{lemma}[thm]{Lemma}

\theoremstyle{definition}
\newtheorem{definition}[thm]{Definition}
\newtheorem{remark}[thm]{Remark}

\newtheorem{example}[thm]{Example}

\makeatletter
\newtheorem*{rep@theorem}{\rep@title}
\newcommand{\newreptheorem}[2]{%
\newenvironment{rep#1}[1]{%
 \def\rep@title{#2 \ref{##1}}%
 \begin{rep@theorem}}%
 {\end{rep@theorem}}}
\makeatother

\newreptheorem{theorem}{Theorem}

\newcommand{\Z}{\ensuremath{\mathbb{Z}}}

\newcounter{nootje}
\numberwithin{equation}{theorem}

\usepackage[margin=1.5in]{geometry}

\begin{document}
\title[Partitions and Class Groups]{Partitions Associated to Class Groups of Imaginary Quadratic Number Fields}
\author{Kathleen L. Petersen}
\author{James A. Sellers}
\address{Mathematics and Statistics Department, University of Minnesota Duluth, Duluth, Minnesota 55812 USA}
\email{kpete@umn.edu, jsellers@umn.edu}


\begin{abstract}  
We investigate properties of  attainable partitions of  integers, where a partition $(n_1,n_2, \dots, n_r)$ of $n$ is attainable if $\sum (3-2i)n_i\geq 0$. Conjecturally, under an extension of the Cohen and Lenstra heuristics by Holmin et. al., these partitions correspond to abelian $p$-groups that appear as class groups of imaginary quadratic number fields for infinitely many odd primes $p$.  We demonstrate a connection to partitions of integers into triangular numbers, construct a generating function for  attainable partitions, and determine the maximal length of attainable partitions.
\end{abstract}

\maketitle

\section{Introduction}

We consider partitions  $\lambda = (n_1, n_2, \dots, n_r)$ of $n$ so that $n=n_1+n_2+\dots +n_r$ and written so that $n_1\geq n_2\geq\dots \geq n_r$.  For a fixed odd prime $p$, we realize the bijection between  partitions of $n$ and abelian groups of order $p^n$ by associating to $\lambda$ the abelian $p$-group 
\[ G_{\lambda}(p) = (\mathbb Z/ p^{n_1} \mathbb Z)  \times (\mathbb Z/ p^{n_2} \mathbb Z) \times  \dots \times (\mathbb Z/ p^{n_r} \mathbb Z).   \]
The  {\bf cyclicity index of $\lambda$},
\[ c(\lambda) = \sum_{i=1}^r (3-2i)n_i\]
measures the deviation of  $G_{\lambda}(p)$ from being cyclic and governs the size of $\mathrm{Aut}(G_{\lambda}(p))$ (see Section~\ref{section:background}). The partition $\lambda$ is called  {\bf attainable}  if $c(\lambda) \geq 0.$ Holmin et. al. \cite{MR3955814} show that under an extension of the Cohen-Lenstra heuristics \cite{MR756082} for the distribution of class numbers, an attainable partition corresponds to a family of abelian $p$-groups that conjecturally are realized as the class groups of imaginary quadratic number fields for infinitely many odd primes $p$.  We study  attainable partitions to further understand which abelian $p$-groups should occur in this way. 

In this work, we determine a generating function   for the attainable partitions.   
 \begin{thm}\label{theorem:attainablepartitionfunction}
The generating function for the attainable partitions of $n$  is 
\[
A(q) =  \frac{1 }{1-q}  \prod_{i=1}^{\infty} \frac{1}{ (1-q^{i(i+1)})}.
\]
\end{thm}
To prove this, we show (Lemma~\ref{lemma:growingneedszero}) that the growth of the number of attainable partitions is governed by partitions of cylicity index 0, and there are no partitions of odd integers with cyclicity index equal to zero. 
Then we demonstrate the following. 
\begin{thm}\label{theorem:triangularbijection} 
For each $m\geq 1$, partitions of $2m$ with cyclicity index 0 are in bijective correspondence with partitions of  $m$ into triangular numbers. 
\end{thm}

We also study the shape of attainable partitions.  In particular,  we demonstrate that the maximal length of an attainable partition of $n$ increases like $\sqrt{n}$, where the length of a partition is the number of terms in the partition ($r$ in the notation above). 
\begin{thm}\label{theorem:longestlength} 
The length of an attainable partition $(n_1,n_2, \dots, n_r)$ of $n$ is at most   $r=\lfloor \tfrac{\sqrt{4n+1}+1}2 \rfloor$.  This bound is  realized by $\lambda = (n-r+1,1,\dots, 1)$ with $0\leq c(\lambda) < \sqrt{4n+1}$.  
\end{thm}
The value of $r$ in Theorem~\ref{theorem:longestlength}  is either $\lfloor \sqrt{n} \rfloor$ or  $\lfloor \sqrt{n} \rfloor+1$.   In Remark~\ref{remark:smallestcyclicity} we note that the smallest cyclicity index of an attainable partition of  $n=2m$  is 0, realized by $(m,m)$, and the smallest cyclicity index of an attainable  partition of  $n=2m+1$  is 1, realized by $(m+1,m)$.

\section{Background}\label{section:background}

 We now give some background to motivate our study of these partitions. 
Based on experimental observations,  Cohen and Lenstra  \cite{MR3955814} noticed that class groups of quadratic number fields behave like a random sequence with respect to a probability distribution on the space of  finite abelian groups.   Their key observation was that the odd part of the class group  is  rarely non-cyclic. The even part of the class group is well understood due to Gauss' genus theory \cite{gauss}. This led to a heuristic assumption that the weighting of isomorphism classes of abelian groups should be inversely proportional to the size of their automorphism group.    
When $G$ is an abelian group of odd order, this can be stated as  $\mathcal{F}(G) \approx P(G)  \mathcal{F}(|G|) $ where $\mathcal{F}(G)$ and $\mathcal{F}(h)$ are the number of fundamental discriminants whose associated class group is $G$, and class number is $h$, respectively. The quantity $P(G)$ is given by 
\[
P(G) =  \frac{1}{| \mathrm{Aut}(G)|} / \Big( \sum \frac{1}{| \mathrm{Aut} (G')|} \Big)
\]
where the sum is over abelian groups $G'$ of order $|G|$.

This leads to a natural question: What finite abelian groups occur as class groups of imaginary quadratic fields?  Chowla \cite{chowla} showed that for large $r$, $(\mathbb{Z}/2\mathbb{Z})^r$ does not occur.  In fact,  $(\mathbb{Z}/n\mathbb{Z})^r$ does not occur for sufficiently large $r$  (depending on $n$) \cite{ MR289454, MR313221, MR2394923}.  However, these results are  not effective and do not yield explicit examples of non-occurence.  Watkins \cite{MR2031415} determined all imaginary quadratics with class number at most 100,  and this work shows, for example, none of the groups 
\[
(\mathbb{Z}/3\mathbb{Z})^3, (\mathbb{Z}/9\mathbb{Z})\times (\mathbb{Z}/3\mathbb{Z})^2,  (\mathbb{Z}/3\mathbb{Z})^4
\]
 occurs as the class group of an imaginary quadratic field (see \cite{MR3955814}).

Holmin et. al. \cite{MR3955814} used the Cohen-Lenstra heuristics to predict the $p$-group decomposition of the class group of imaginary quadratics. 
The automorphisms of $G_{\lambda}(p)$ are intrinsically tied to $c(\lambda)$ since 
\[
|\mathrm{Aut}(G_{\lambda}(p)) |= p^{2n-c(\lambda)} \prod_{i=1}^k \prod_{j=1}^{m_i} \Big( 1-\frac{1}{p^j} \Big)
\] 
where $k$ is the number of distinct parts of $\lambda$ and $m_i$ is the multiplicity of the $i$th part (see  \cite{MR1500775}). 
Holmin et al. establish \cite[Proposition 7.1]{MR3955814} that 
\[
P(G_{\lambda}(p))  \sim p^{c(\lambda)-n}
\]
which with the Cohen-Lenstra heuristics leads to their  Conjecture 1.2 that if  $c(\lambda)>0$ then as $p\rightarrow \infty$  
\[
\mathcal{F}(G_{\lambda}(p)) \sim \frac{\frak C}{n} \cdot \frac{p^{c(\lambda)}}{\log p}
\]
where  $\frak C\sim 11.317$ is a constant.  Therefore, the expectation is that when $c(\lambda)>0$ the group $G_{\lambda}(p)$  appears as a class group for all but finitely many primes $p$. 
 For $c(\lambda)=0$ the conjecture states that as $x\rightarrow \infty$ then 
\[ \sum_{p\leq x}  \mathcal{F}(G_{\lambda}(p))  \sim  \frac{\frak C}{n} \frac{x}{(\log x)^2}. \]
So it is expected that $G_{\lambda}(p)$  occurs for infinitely many primes  $p$ and also does  not occur for infinitely many primes $p$. 
For $c(\lambda)<0$ the conjecture states that $G_{\lambda}(p)$  appears as a class group for only finitely many $p$.  
Computations determining the number of imaginary quadratics with prescribed odd $p$-group appearing as class groups of order up to $10^6$ (conditional on the GRH) in \cite{MR3955814} support this conjecture.

 In light of Theorem~\ref{theorem:longestlength}, these conjectures suggest that  there are infinitely many primes $p$ such that the group
\[
( \mathbb{Z}/ p^{n-\lfloor \sqrt{n} \rfloor+1} \mathbb{Z}) \times (\mathbb{Z}/p \mathbb{Z})^{\lfloor \sqrt{n} \rfloor-1}
\] 
appears as the class group of an imaginary quadratic number field. 
The partitions  $(m,m)$ of $2m$ and $(m+1,m)$ of $2m+1$ of smallest cyclicity  index mentioned in Remark~\ref{remark:smallestcyclicity}  correspond  to the groups $(\Z/p^m\Z)^2$ and $(\Z/p^{m+1}\Z) \times (\Z/p^m\Z)$,  with automorphism groups of order $ p^{2n-3}(p+1)(p-1)^2$ and $p^{2n-3}(p-1)^2$, respectively.  These partitions correspond to the largest possible automorphism group of an abelian $p$-group of order $p^n$ associated to an attainable partition.  By way of comparison, $(n)$ has cyclicity index $n$ and corresponds to the cyclic group $\Z/p^n\Z$ which has $p^{n-1}(p-1)$ automorphisms.

Data from the aforementioned computations are available online (see  \cite{HKnoncylic}, \cite{HKorders}, and \cite{HKdisc}).  For $n=5$, the odd primes with $p^5<10^6$ are $p=3,5,7, 11, 13$.  The only partition of negative cyclicity index that is realized for these primes is $\lambda = (3,1,1)$ with $c(\lambda) = -1$.  This is realized for $p=5$; the group $\Z/5^3 \Z \times (\Z/5\Z)^2$ is the class group for the imaginary quadratic with fundamental discriminant $145367147$.

\begin{table}[h!]
  \begin{center}
    \label{tab:tablenis6}
    \begin{tabular}{cccccccc} 
      $\lambda$ &  $G_{\lambda}(p)$   & $c(\lambda)$ &  $p=$ 3 & 5&7 & 11 & 13   \\
      \hline
	(5) &   $\Z/p^5 \Z $ & 5 & 549 & 4610 & 19430 & 147009 & 314328 \\
	(4,1) & $\Z/p^4\Z \times \Z/p\Z$  & 3  & 56 & 218 & 444 & 1347 &  1894 \\
	(3,2) & $\Z/p^3\Z \times \Z/p^2\Z$ & 1 & 8 &  5 & 8 & 13 &  9 
    \end{tabular}
            \caption{{Attainable partitions of $5$ with corresponding group, cyclicity index and, for $p\leq 13$, the values of $\mathcal{F}(G_{\lambda}(p))$}}
  \end{center}
\end{table}

For $n=6$  no non-attainable partitions are realized for $p^6<10^6$. 

\begin{table}[h!]
  \begin{center}
    \label{tab:tablenis6}
    \begin{tabular}{cccccc} 
      $\lambda$ &  $G_{\lambda}(p)$   & $c(\lambda)$ & $p=3$ & $5$ & $7$   \\
      \hline
	(6) &   $\Z/p^6 \Z $ & 6 & 1512 & 19469 & 116278 \\
	(5,1) & $\Z/p^5\Z \times \Z/p\Z$  &  4 & 177 & 1024 & 2887 \\
	(4,2) & $\Z/p^2\Z \times \Z/p^2\Z$ & 2 & 18 &  37 & 58 \\
	(3,3) & $\Z/p^3\Z \times \Z/p^3\Z$  & 0 & 2 & 2 & 3 \\
	(4,1,1) & $\Z/p^4\Z \times (\Z/p\Z)^2$ & 0 & 0 & 3 & 0 
    \end{tabular}
    \caption{{Attainable partitions of $6$ with corresponding group, cyclicity index and, for $p\leq 7$, the values of $\mathcal{F}(G_{\lambda}(p))$}}
  \end{center}
\end{table}

\section{Preliminaries}

 As mentioned in \cite{MR3955814} if $\lambda$ is a partition of $n$ then $1-(n-1)^2 \leq c(\lambda) \leq n$, and $c(\lambda)=n$ exactly when $G_{\lambda}(p)$ is cyclic.  These extremal partitions of $n$ are readily seen to be $\lambda_1=(1,1,\dots, 1)$ with $c(\lambda_1)= 1-(n-1)^2$ and $\lambda_2= (n)$ with $c(\lambda_2) = n$, and  correspond to the $p$-groups $G_{\lambda_1}(p)= (\mathbb{Z}/p\mathbb{Z})^n$ and $G_{\lambda_2}(p)=\mathbb{Z}/p^n \mathbb{Z}$.

We now transition to a a more detailed study of the properties of attainable partitions.
\begin{lemma}\label{lemma:weightparity}
For any partition  $\lambda$   of $n$ we have $c(\lambda)\equiv n \pmod 2$. 
\end{lemma}

\begin{proof}
Let $\lambda=(n_1,n_2, \dots, n_r)$ be a partition of $n$.  By definition    
\[
c(\lambda)  = \sum_{i=1}^r (3-2i)n_i \equiv  \sum_{i=1}^r n_i \pmod 2.
\]
The result follows as $ n_1+n_2 + \dots + n_r=n$. 
\end{proof}

Thus, if $n$ is odd then $c(\lambda)$ cannot be zero,   so there are no   partitions of odd $n$ with cyclicity index equal to zero.

\begin{definition}
 Given a partition $\lambda=(n_1,n_2,\dots,  n_r)$ of $n$ as above, we say that the partition $\lambda'$ of $n+1$ given by $\lambda'=(n_1+1, n_2, \dots, n_r)$ is a {\bf primary addition} to $\lambda$.  If $n_1 \neq n_2$ then we say $\lambda'' =( n_1, n_2+1, n_3, \dots, n_r)$ is a {\bf secondary addition} to $\lambda$.
\end{definition}

The partition $\lambda' = (8,1,1)$ of 10 is a primary addition to the partition $\lambda=(7,1,1)$ of 9, and the partition $\lambda''=(7,2,1)$ of 10 is a secondary addition to $\lambda$. We calculate $c(\lambda)=3$, $c(\lambda')=4$ and $c(\lambda'')=2$.   
\begin{lemma}  \

\begin{itemize}\label{lemma:type1and2}
\item If $\lambda'$ is a primary addition to  $\lambda$ then $c(\lambda')=c(\lambda)+1.$
\item If $\lambda''$ is a secondary addition to  $\lambda$ then $c(\lambda')=c(\lambda)-1.$
\end{itemize}
\end{lemma} 
 
\begin{proof}
For the first assertion, 
\[ c(\lambda) = \sum_{i=1}^r (3-2i)n_i = n_1 + \sum_{i=2}^r (3-2i)n_i \]
and  so 
\[ c(\lambda) +1 = (n_1+1) +  \sum_{i=2}^r (3-2i)n_i = c(\lambda').\]
The second assertion follows similarly.

\end{proof} 
 
  This implies that if $\lambda$ is an attainable partition of $n$ then  a primary addition to $\lambda$ is an attainable partition of $n+1$.   Moreover,  if $c(\lambda)=0$ then $\lambda$ cannot be obtained using a secondary addition.  The following follows directly from Lemma~\ref{lemma:weightparity} and Lemma~\ref{lemma:type1and2}.

\begin{remark}\label{remark:smallestcyclicity}
The smallest cyclicity index of an attainable  partition of  $n=2m$  is 0, realized by $(m,m)$.  The smallest cyclicity index of  an attainable partition of  $n=2m+1$  is 1, realized by $(m+1,m)$.  
\end{remark}

\begin{table}[h!]
  \begin{center}
    \label{tab:table1}
    \begin{tabular}{c|l} 
      $n$ & \qquad \qquad \qquad \qquad   Attainable Partitions of $n$   \\
      \hline
      1 & $(1)$ \\
      2 & $(3), (2,1)$ \\
      3 & $(3), (2,1)$\\
      4 & $ (4), (3,1), (2,2)$ \\
      5 & $(5), (4,1), (3,2)$ \\
      6 & $(6), (5,1), (4,2), (3,3), (4,1,1)$\\
      7 & $(7), (6,1), (5,2), (4,3), (5,1,1)$\\
      8 & $(8), (7,1), (6,2), (5,3), (4,4), (6,1,1), (5,2,1)$\\
      9 & $(9), (8,1), (7,2), (6,3), (5,4), (7,1,1), (6,2,1)$ \\
      10 & $(10), (9,1), (8,2), (7,3), (6,4), (5, 5), (8,1,1), (7,2,1), (6,3,1)$ \\
      11 & $(11), (10,1), (9,2), (8,3), (7,4), (6,5), (9,1,1), (8,2,1), (7,3,1)$ \\
      12 & $(12), (11,1), (10,2), (9,3), (8,4), (7,5), (6, 6), (10,1,1), (9,2,1), (8,3,1), (7,4,1),$ \\
      		& $  (8,2,2), (9,1,1,1)$ \\
      13 & $(13), (12,1), (11,2), (10,3), (9,4), (8,5), (7,6), (11,1,1), (10,2,1), (9,3,1),   (8,4,1), $ \\
      		& $  (9,2,2), (10,1,1,1)$ \\
      14 & $(14), (13,1), (12,2), (11,3), (10,4), (9,5), (8,6), (7, 7), (12,1,1), (11,2,1),   (10,3,1),  $\\ 
      		& $ (9,4,1), (8,5,1), (10,2,2), (9,3,2), (11,1,1,1), (10,2,1,1)   $ \\
      15 & $(15), (14,1), (13,2), (12,3), (11,4), (10,5), (9,6), (8,7), (13,1,1), (12,2,1),  (11,3,1),  $\\
      		&  $ (10,4,1), (9,5,1), (11,2,2), (10,3,2), (12,1,1,1),  (11,2,1,1) $ 
    \end{tabular}
    \caption{{All attainable partitions of $n$ for $n\leq 15$}}
  \end{center}
\end{table}

\begin{definition}
Let $a(n)$ denote the number of attainable partitions of $n$. 
Let $z_0(n)$ denote the number of partitions of $n$ with cyclicity index equal to zero, and for an even number $2m$ let $z(m)=z_0(2m)$. 
We set $a(0)=z_0(0)=1$.
\end{definition}
It follows from Lemma~\ref{lemma:weightparity} that $z_0(2m+1)=0$ as there are no partitions of an odd number with cyclicity index 0.

\begin{lemma}\label{lemma:growingneedszero} For any natural number $n$, 
\[
a(n+1)= a(n) + z_0(n+1).
\]
\end{lemma}

\begin{proof}

Partitions of $n+1$ obtained from partitions of $n$ by primary additions have positive  cyclicity index by Lemma~\ref{lemma:type1and2}. As such, it is enough to show that if $\lambda$ is a partition of $n+1$ and $c(\lambda) > 0$, then $\lambda$ can be obtained from an attainable partition of $n$ by a primary addition. 
Let $\lambda = (n_1,n_2, \dots, n_r)$ be a partition of $n+1$ and since $c(\lambda)>0$ we have $n_1\neq n_2$. Then $\lambda'= (n_1-1,n_2, \dots, n_r)$ is a partition of $n$ with $c(\lambda')=c(\lambda)-1$ by Lemma~\ref{lemma:type1and2}.  Since $c(\lambda)>0$ we have $c(\lambda')\geq 0$ is an attainable partition of $n$. 
\end{proof}

\begin{lemma}\label{lemma:oddsareeasy}
 For $m\geq 1$ we have $a(2m+1)=a(2m)$ and
all attainable partitions of $2m+1$ occur as primary additions to a partition of $2m$. 
\end{lemma}

\begin{proof}
 Lemma~\ref{lemma:weightparity} implies that $z_0(2m+1)=0$ and by  Lemma~\ref{lemma:growingneedszero} we conclude $a(2m+1)=a(2m)$. 
By Lemma~\ref{lemma:type1and2} every attainable partition of $2m$ yields an attainable partition of $2m+1$ by a primary addition.  
\end{proof}

\begin{lemma}\label{lemma:moreinevens}
For all $m\geq 1$ we have $a(2m+2)>a(2m+1)$.
\end{lemma}

\begin{proof}
Since every primary addition  to an attainable partition of $2m+1$ is an attainable partition of $2m+2$ we have that $a(2m+2)\geq a(2m+1)$.  
To show that the inequality is strict, consider the attainable partition $\lambda''=(m+1,m+1)$ of $2m+2$. Since $c(\lambda'')=0$ it is not a primary addition to any attainable partition of $2m+1$ by Lemma~\ref{lemma:type1and2}. (In fact, it is a secondary addition to $(m+1,m)$.)

\end{proof}

Lemma~\ref{lemma:moreinevens} implies that  there are always attainable partitions of $2m+2$ that are not primary additions to a partition of $2m+1$. 
We see from the data that we often get new attainable partitions by increasing the partition length, for example the attainable partition $(4,1,1)$ of $6$ has length 3, but there are no attainable length 3 partitions of $n$ for $1\leq n \leq 5$.

\section{Lengths of Attainable Partitions}

\begin{lemma}\label{lemma:boundonnk}
If $\lambda=(n_1,  \dots, n_r)$ is attainable then, for $1< k\leq r$, $n_k\leq n/k(k-1)$.
\end{lemma}

\begin{proof}
We have that $n=n_1+n_2+\dots + n_r$ so that 
\begin{align*}
n & \geq n_1+n_2+\dots + n_k \\
& \geq (\sum_{i=2}^{k} (2i-3)n_i) + (n_2+ \dots +n_k) \\
& \geq \sum_{i=2}^k (2i-2)n_k =  k(k-1) n_k
\end{align*}
where we have used the fact that $\lambda$ is attainable so $n_1\geq n_2+3n_3+5n_4+ \dots + (2r-3)n_r \geq n_2+3n_3+5n_4+ \dots + (2k-3)n_k$ and the fact that if $i<j$ then $n_i\geq n_j$.

\end{proof}

\begin{reptheorem}{theorem:longestlength}
{\it The length of an attainable partition $(n_1,n_2, \dots, n_r)$ of $n$ is at most   $r=\lfloor \tfrac{\sqrt{4n+1}+1}2 \rfloor$.  This bound is  realized by $\lambda = (n-r+1,1,\dots, 1)$ with $0\leq c(\lambda) < \sqrt{4n+1}$.   } 
\end{reptheorem}

\begin{proof}

Lemma~\ref{lemma:boundonnk}  implies that if $n/(k(k-1)) < 1$ then $n_k=0$. Therefore, the partition length, $r$, of an attainable partition is the greatest  $r\in \mathbb{Z}$ such that $n\geq r(r-1)$. Completing the square, we see that 
\[ r=\lfloor \sqrt{n+\tfrac14}+\tfrac12 \rfloor=\lfloor \tfrac{\sqrt{4n+1}+1}2 \rfloor. \] 
Due to the weighting of $c(\lambda)$ length is maximized by partitions of the above shape. 

The cyclicity index bound is computed directly using the inequality $a-1<\lfloor a \rfloor \leq a$ with the definition of $c(\lambda)$.
\end{proof}
This length bound may be achievable by multiple partitions as evidenced by the length 4 partitions $(11,1,1,1)$ and  $(10,2,1,1)$ of 14.  The partition of the form stated in Theorem~\ref{theorem:longestlength} does not always have the smallest cyclicity index; $c((11,1,1,1))=2$ and $c((10,2,1,1))=0$.

 It is elementary to verify that the $r$ value in the above satisfies $\lfloor \sqrt{n} \rfloor \leq r\leq \lfloor \sqrt{n} \rfloor +1$.  This lower bound is achieved by perfect squares, and in general any $n$ where the fractional part of $\sqrt{n}+\tfrac14$ is less than $\tfrac12$. The upper bound for $n$ where $\sqrt{n}+\tfrac14$ has fractional part greater than or equal to $\tfrac12$, for example integers one less than a perfect square.

\section{Generating Functions}\label{section:generatingfunctions}

In this section we prove Theorem~\ref{theorem:attainablepartitionfunction}. First, we prove Theorem~\ref{theorem:triangularbijection} which establishes a connection between partitions of cyclicity index 0 and partitions of integers into triangular numbers.

From Lemma~\ref{lemma:growingneedszero} and Lemma~\ref{lemma:oddsareeasy} we have that 
\begin{align*}
a(2m) 
& = a(2m-2) + z(m) \\
& = a(2m-4) + z(m-1)+z(m) \\ 
&  \quad \vdots \\
& = a(2) + z(2)+z(3) + \dots + z(m).
\end{align*}
Since $a(2)=2$ and $z(1)=1$  we have that $a(2)=a(0)+z(1)$ and we have shown the following.
\begin{lemma}\label{lemma:a(2m)sum}
For all $m\geq 1$, $ \displaystyle a(2m) = \sum_{r=0}^m z(r).$
\end{lemma}

\begin{table}[h!]
  \begin{center}
    \label{tab:table1}
    \begin{tabular}{c|l} 
      $m$ & Partitions of $2m$ with Cyclicity Index  0  \\
      \hline
      1 & $(1,1) $ \\
      2 & $ (2,2) $ \\
      3 & $ (3,3), (4,1,1) $\\
      4 & $(4,4), (5,2,1) $\\
      5 & $ (5, 5),  (6,3,1)$ \\
      6 & $ (6, 6), (7,4,1), (8,2,2), (9,1,1,1) $ \\
      7 & $ (7, 7), (8,5,1),(9,3,2), (10,2,1,1) $ 
      \end{tabular}
    \caption{{All partitions of $2m$ with cyclicity index 0 for $m\leq 7$}}
  \end{center}
\end{table}

\begin{definition}
 We will write $t_i=i(i+2)/2$ for the $i^{th}$ {\bf triangular number}. 
We call the numbers of the form $2t_i$ the {\bf oblong numbers}. 
\end{definition}

Now we prove that  partitions of $2m$ with cyclicity index 0 are in bijective correspondence with partitions of $2m$ into oblong numbers, which is equivalent to the following. 
 
\begin{reptheorem}{theorem:triangularbijection}{
 For each $m\geq 1$, partitions of $2m$ with cyclicity index 0 are in bijective correspondence with partitions of $m$ into triangular numbers. 
}
\end{reptheorem}

\begin{proof}
Let $\lambda=(n_1,n_2, \dots, n_r)$ be a partition of $2m$ with $c(\lambda)=0$.  Using the definition of $c(\lambda)$,  
\[
n_1 = n_2+3n_3+5n_4+\dots (2k-3)n_r
\]
and we can write
\begin{align*}
2m
& = n_1+n_2+n_3+n_4+ \dots +n_r\\
& = (n_2+3n_3+5n_4+\dots (2r-3)n_r) + n_2+n_3+n_4+ \dots +n_r \\
& = 2n_2 +4n_3 + 6n_4+\dots + (2r-2)n_r. 
\end{align*}
We rewrite this as 
\begin{align*}
2m & = 2(n_2-n_3) + 6(n_3-n_4) + 12(n_4-n_5) + 20(n_5-n_6)  \\
& \quad + \dots +  2t_{r-2}(n_{r-1}-n_r)+2t_{r-1}n_r.
\end{align*}
Since $n_i\geq n_{i+1}$ the $n_i-n_{i-1}$ terms are all non-negative, and the final term is non-zero as $n_r>0$.    This demonstrates that $2m$ is a sum of non-negative multiples of oblong numbers.  Dividing both sides of the equation by 2 demonstrates that $m$ is a sum of non-negative multiples of triangular numbers. 

Conversely, assume that  $n=c_1t_1+ \dots + c_{r-1}t_{r-1}$ with $c_i\geq 0$ for $i=1, \dots, r-2$ and $c_{r-1}>0$. We let  $n_r=c_{r-1}$,  and for $i=2, \dots r-1$ let
\[
n_i = c_{i-1}+c_i + \dots + c_{r-1}
\]
and  define $n_1=n_2+3n_3+5n_4+\dots (2r-3)n_r$.  As such, $n_i>0$ and $n_{i+1}\geq n_i$ for $i=1, \dots, r-1$ so that $\lambda = (n_1, n_2, \dots, n_r)$ is a partition of $2n$ with $c(\lambda)=0$.

\end{proof}

\begin{table}[h!]
  \begin{center}
    \label{tab:table1}
    \begin{tabular}{c|l} 
      $n$ & Partitions of $n$ into Triangular Numbers  \\
      \hline
      1 & $(1)$ \\
      2 & $ (1,1)$ \\
      3 & $ (3), (1,1,1) $\\
      4 & $(3,1), (1,1,1,1)$\\
      5 & $(3,1,1),  (1,1,1,1,1)$ \\
      6 & $ (6), (3,3), (3,1,1,1), (1,1,1,1,1,1) $ \\
      7 & $ (6,1), (3,3,1),(3,1,1,1,1), (1,1,1,1,1,1,1)$ 
      \end{tabular}
          \caption{{All partitions of $n$  into triangular numbers for $n\leq 7$}}
  \end{center}
\end{table}

\begin{example}
Consider the partition $(8,5,1)$ of 14, which has cyclicity index 0.  We have 
\[ 14=8+5+1=   (5+3 \cdot 1)+5+1 = 2\cdot 5 + 4 \cdot 1=  2(5-1)+6(1-0)\] 
 which corresponds to writing $7=1\cdot 4+3\cdot 1$ giving the summation $7=1+1+1+1+3$ in terms of triangular numbers. 

Now consider $7=1+3+3$ as a different sum of triangular numbers. Writing $t_1=1$ as the first triangular number and $t_2=3$ as the second, in the notation above $7=c_1t_1+c_2t_2$ with $c_1=1$ and $c_2=2$. We have $n_3=c_2=2$, $n_2=c_1+c_2=3$ and $n_1=n_2+3n_3=9$.  This gives us the partition $(9,3,2)$ of 14 with cyclicity index 0.  

In this correspondence the partition $(1,1,1,\dots, 1)$ of $n$ into triangular numbers corresponds to the partition $(n,n)$ of $2n$ with cyclicity index 0. 
\end{example}

\begin{cor}\label{cor:zgenerating}
The generating function for $z(n)$ is 
\[ 
\sum_{n=0}^{\infty} z(n)q^n =  \prod_{i=1}^{\infty} \frac{1}{1-q^{i(i+1)/2}}. 
\]

\end{cor}

\begin{proof}
By Theorem~\ref{theorem:triangularbijection} it suffices to determine a generating function for partitions into triangular numbers, which is the generating function above.



\end{proof}

Since $z_0(2m)=z(m)$, by Corollary~\ref{cor:zgenerating},   the generating function for $z_0(2m)$ is 
\[
\sum_{m=0}^{\infty} z_0(2m) q^{2m} = \prod_{i= 1}^{\infty}  \frac{1}{  (1-q^{i(i+1)})}.
\]
By Lemma~\ref{lemma:a(2m)sum} we have 
\[ a(2m) = \sum_{j=0}^m z(j) = \sum_{j=0}^m z_0(2j)   \]
and the generating function of $a(2m)$ is 
\[
\sum_{m=0}^{\infty} a(2m)q^{2m}  = \frac{1}{(1-q^2)} \prod_{i= 1}^{\infty}  \frac{1}{  (1-q^{i(i+1)})}.
\]

Because $a(2m+1)=a(2m)$ the generating function for $a(2m+1)$ is $\sum_{m=0}^{\infty} a(2m)q^{2m+1}$.
Putting this together, the generating function for $a(n)$ has the form  
\begin{align*} \sum_{n=0}^{\infty} a(n) q^n & = \sum_{m=0}^{\infty} a(2m) q^{2m} +  \sum_{m=0}^{\infty}  a(2m+1)q^{2m+1} \\
& = (1+q) \sum_{m=0}^{\infty}  a(2m)q^{2m}. 
\end{align*}
From above we have 
\[ 
\sum_{n=0}^{\infty} a(n)  q^n   = \frac{(1+q)}{(1-q^2) } \prod_{i=1}^{\infty}   \frac{1 }{  (1-q^{i(i+1)})}=\frac{1}{1-q } \prod_{i=1}^{\infty}   \frac{1 }{  (1-q^{i(i+1)})}  
\]
proving Theorem~\ref{theorem:attainablepartitionfunction}.

\bibliographystyle{amsplain}
\bibliography{attainablebib}

\end{document}